\documentclass{math}
\usepackage{hyperref}
\hypersetup{
colorlinks=true,
urlcolor=blue,
citecolor=blue}
\usepackage[all]{xy,xypic}
\usepackage{amsfonts,amssymb,amsmath,amsgen,amsopn,amsbsy,theorem,graphicx,epsfig}
\usepackage{eufrak,amscd,bezier,latexsym,mathrsfs,enumerate,multirow}
\usepackage[utf8]{inputenc}\usepackage[english]{babel}
\usepackage[dvipsnames]{xcolor}
\usepackage[pagewise]{lineno}
\linenumbers

\yil{}
\vol{}
\fpage{}
\lpage{}
\doi{}

\def\Gr{\mathop{\rm Gr}}

\def\P{{\mathcal P}}

\def\M{{\mathcal M}}

\def\bnu{{\boldsymbol \nu}}

\def\tpi{{\tilde \pi}}

\def\sX{{\mathsf X}}

\def\sA{{\mathsf A}}

\def\sE{{\mathsf E}}

\theoremstyle{remark}
\newtheorem{assumption}{Assumption}

\allowdisplaybreaks

\title{Discrete-time average-cost mean-field games on Polish spaces}

\author[AUTHOR]{
\textbf{Naci Saldi$^{1}$\thanks{Correspondence: naci.saldi@ozyegin.edu.tr}} \\
$^{1}$Department of Natural and Mathematical Sciences, Faculty of Engineering, Özyeğin University, İstanbul, Turkey, \\
ORCID iD: https://orcid.org/0000-0002-2677-7366, corresponding author.
\\ [1.8em]

\rec{.201}
\acc{.201}
\finv{..201}
}

\amssayisi{2010 {\itshape AMS Mathematics Subject Classification:} 91A15, 91A10, 91A13, 93E20. 
\newline
\vspace{-15mm}
\begin{center}
		 \raisebox{-17ex}[0ex][0ex]{~~ \raisebox{.5ex}[0ex][0ex]{\footnotesize  This work is licensed under a Creative Commons Attribution 4.0 International License.}}
 		    \end{center}}

\newcommand{\bc}{\begin{center}}
\newcommand{\ec}{\end{center}}

\numberwithin{equation}{section}

\newtheorem{theorem}{Theorem}[section]

\newtheorem{definition}[theorem]{Definition}

\newtheorem{proposition}[theorem]{Proposition}

\nolinenumbers
\begin{document}

\maketitle

\begin{abstract}
In stochastic dynamic games, when the number of players is sufficiently large and the interactions between agents depend on empirical state distribution, one way to approximate the original game is to introduce infinite-population limit of the problem. In the infinite population limit, a generic agent is faced with a \emph{so-called }mean-field game. In this paper, we study discrete-time mean-field games with average-cost criteria. Using average cost optimality equation and Kakutani's fixed point theorem, we establish the existence of Nash equilibria for mean-field games under drift and minorization conditions on the dynamics of each agent. Then, we show that the equilibrium policy in the mean-field game, when adopted by each agent, is an approximate Nash equilibrium for the corresponding finite-agent game with sufficiently many agents.  
\end{abstract}

\section{Introduction}\label{sec1}

In this paper, we consider discrete-time mean-field games subject to average-cost criteria with Polish state and action spaces. These games arise as the infinite population limit of finite-agent dynamic games, where agents interact through the empirical distribution of their states. The main goal in mean-field games is to establish the existence of an optimal policy and a state distribution that are consistent with each other. Then, this optimal policy, when adopted by each agent, forms an approximate Nash equilibrium for finite-agent games with sufficiently many agents.

Mean-field games have been introduced by Huang, Malham\'{e}, and Caines \cite{HuMaCa06} and Lasry and Lions \cite{LaLi07} to approximate continuous-time differential games with a large but finite number of identical agents interacting with each other via empirical distribution of their states (i.e., mean-field term). The key feature of this approach is to transform the game problem into a non-classical stochastic control problem by passing to the infinite-population limit. In the infinite population limit, since empirical state distribution converges to a deterministic probability measure by the law of large numbers, agents are decoupled from each other and each agent is faced with a stochastic control problem that has a constraint on the distribution of its state. The latter problem is called \emph{mean-field game} in the literature. The optimal solution of this stochastic  control problem provides an approximate Nash equilibrium when the number of agents is sufficiently large. In continuous-time differential games, this optimal solution is characterized by a Fokker-Planck equation evolving forward in time and a Hamilton-Jacobi-Bellman equation evolving backward in time. We refer the reader to
 \cite{HuCaMa07,TeZhBa14,Hua10,BeFrPh13,Ca11,CaDe13,GoSa14,MoBa16} for studies of continuous-time mean-field games with different models and cost functions, such as games with major-minor players, risk-sensitive games, games with Markov jump parameters, and LQG games.

Discrete-time mean-field games have not been studied much in the literature.
Existing works have mostly studied games with finite or countable state spaces  subject to finite-horizon or infinite-horizon discounted cost criteria. \cite{GoMoSo10} considers a discrete-time mean-field game with a finite state space over a finite horizon. In \cite{AdJoWe15}, discrete-time mean-field game with countable state-space is studied subject to an infinite-horizon discounted cost criterion. References \cite{ElLiNi13,MoBa15,NoNa13,MoBa16-cdc} consider discrete-time mean-field games with linear state dynamics. There are only three papers \cite{Bis15,Wie19,WiAl05} studying discrete-time mean-field games subject to average cost criteria. In \cite{WiAl05}, authors consider discrete set-up for average-cost mean-field games. In \cite{Bis15}, the author considers average-cost mean-field games with $\sigma$-compact Polish state spaces. In that paper, it was assumed that, for the finite agent game problem, the dynamics of agents do not depend on the empirical distribution of the states. Under strong regularity conditions on system components, \cite{Bis15} established the existence of  Nash equilibria for finite-agent games, and then, showed that these Nash equilibria converge to mean-field equilibria in the infinite-population limit. These imposed regularity conditions are in general prohibitive because they are stated in terms of a specific metric topology on the set of policies, and appear to be too strong to hold under reasonable assumptions. \cite{Wie19} considers average-cost mean-field games with compact state spaces. This setup is the closest to the one studied in this paper. However, in addition to the state space being compact, \cite{Wie19} also uses a completely different technique to establish the existence of equilibrium in the infinite-population limit. Namely, \cite{Wie19} proves the existence of mean-field equilibrium using ergodic properties of Markov chains induced by policies whereas we employ here the dynamic programming principle, stated through average cost equation, to establish the existence of mean-field equilibrium.

In \cite{SaBaRaSIAM} we studied infinite-horizon discounted-cost version of the same problem. Under mild assumptions on the system components, we first established the existence of mean-field equilibrium and then proved that the equilibrium in the mean-field game constitutes an approximate Nash equilibrium for finite-agent games when the number of agents is sufficiently large. However, in the average-cost set-up, it is infeasible to establish similar results under similar assumptions as analysis of average-cost criterion is much more difficult than discounted-cost criterion. Therefore, to establish similar existence and approximation results, we impose drift and minorization conditions on the system dynamics of each agent, which are a bit strong and but are quite common to study average-cost stochastic control problems. 

The paper is organized as follows. In Section~\ref{sec3}, we introduce the infinite-population mean-field game and define mean-field equilibrium. In Section~\ref{sec2}, we formulate the finite-agent game problem of the mean-field type. In Section~\ref{main-proof}, we prove the existence of a mean-field equilibrium. In Section~\ref{sec4-1} we establish that the mean-field equilibrium policies lead to an approximate Nash equilibrium for finite-agent games with sufficiently many agents. Section~\ref{conc} concludes the paper.

\noindent\textbf{Notation.} For a metric space $\sE$, we let $C_b(\sE)$ denote the set of all bounded continuous real functions on $\sE$ endowed with sup-norm $\|g\| = \sup_{e \in \sE} |g(e)|$, which turns $C_b(\sE)$ into a Banach space. Let $\P(\sE)$ denote the set of all Borel probability measures on $\sE$. A sequence $\{\mu_n\}$ of measures on $\sE$ is said to converge weakly to a measure $\mu$ if $\int_{\sE} g(e)\, \mu_n(de)\rightarrow\int_{\sE} g(e) \, \mu(de)$ for all $g \in C_b(\sE)$. We endow $\P(\sE)$ with weak topology induced by weak convergence of probability measures. This topology is known to be metrizable, and if $\sE$ is complete and separable, then $\P(\sE)$ is also complete and separable under weak topology. For metric spaces $\sE_1$ and $\sE_2$, a stochastic kernel $\gamma(\,\cdot\,|e_1)$ (or \emph{regular conditional probability measure}) on $\sE_2$ given $\sE_1$ is a measurable function $\gamma: \sE_1 \rightarrow \P(\sE_2)$. A probability measure $\mu$ is called an invariant probability measure of a stochastic kernel $\gamma$ on $\sE$ given $\sE$ if $\mu(\,\cdot\,) = \int_{\sE} \gamma(\,\cdot\,|e) \, \mu(de)$. For any subset $B$ of $\sE$, we let $\partial B$ and $B^c$ denote the boundary and complement of $B$, respectively. The notation $v\sim \nu$ means that the random element $v$ has distribution $\nu$. Unless otherwise specified, the term ``measurable" will refer to Borel measurability.

\section{Mean-field games and mean-field equilibria}\label{sec3}

The discrete-time mean-field game model is specified by $\bigl( \sX, \sA, p, c, \mu_0 \bigr)$, where $\sX$ and $\sA$ are the Polish (complete and separable metric space) state and action spaces, respectively. The stochastic kernel $p : \sX \times \sA \times \P(\sX) \to \P(\sX)$ denotes the transition probability law of the next state given the previous state-action pair and state-measure. The measurable function $c: \sX \times \sA \times \P(\sX) \rightarrow [0,\infty)$ is the one-stage cost function. The probability measure $\mu_0$ denotes the initial distribution of the state. A policy $\pi$ is a stochastic kernel on $\sA$ given $\sX$; that is, $\pi:\sX \rightarrow \P(\sA)$ is a measurable function, where $\P(\sA)$ is endowed with the Borel $\sigma$-algebra generated by the weak convergence of probability measures. Let $\Pi$ denote the set of all policies. According to the Ionescu Tulcea Theorem \cite{HeLa96}, an initial distribution $\mu_0$, a policy $\pi$, and a transition probability $p$ define a unique probability measure $P^{\pi}$ on $(\sX \times \sA)^{\infty}$. The expectation with respect to $P^{\pi}$ is denoted by $E^{\pi}$.

Given any state-measure $\mu \in \P(\sX)$, a policy $\pi^{*} \in \Pi$ is optimal for $\mu$ if
\begin{align}
J_{\mu}(\pi^{*}) = \inf_{\pi \in \Pi} J_{\mu}(\pi), \nonumber
\end{align}
where
\begin{align}
J_{\mu}(\pi) &= \limsup_{T \rightarrow \infty} \frac{1}{T} E^{\pi}\biggl[ \sum_{t=0}^{T-1} c(x(t),a(t),\mu) \biggr] \nonumber
\end{align}
is the average cost of policy $\pi$ with state-measure $\mu$. In this case, the evolution of the states and actions is given by
\begin{align}
x(0) &\sim \mu_0, \,\,\,
x(t) \sim p(\,\cdot\,|x(t-1),a(t-1),\mu), \text{ } t\geq1, \nonumber \\
a(t) &\sim \pi(\,\cdot\,|x(t)), \text{ } t\geq0. \nonumber
\end{align}
Define the set-valued mapping $\Psi : \P(\sX) \rightarrow 2^{\Pi}$  as $\Psi(\mu) = \{\pi \in \Pi: \pi \text{ is optimal for }  \mu \text{ }\text{ and } \text{ } \mu_0 = \mu\}$.

Conversely, we define another set-valued mapping $\Lambda : \Pi \to 2^{\P(\sX)}$ as follows: given $\pi \in \Pi$, the state-measure $\mu_{\pi}$ is in $\Lambda(\pi)$ if it is a fixed point of the following equation:
\begin{align}
\mu_{\pi}(\,\cdot\,) = \int_{\sX \times \sA} p(\,\cdot\,|x,a,\mu_{\pi})  \, \pi(da|x) \, \mu_{\pi}(dx). \nonumber
\end{align}
Under Assumption~\ref{as1}, which is given below, one can prove that $\Lambda(\pi)$ has a unique element for all $\pi$. Therefore, it is indeed a single-valued mapping.

The notion of an equilibrium for the average-cost mean-field game is defined via these mappings $\Psi$, $\Lambda$ as follows.

\begin{definition}
A pair $(\pi,\mu) \in \Pi \times \P(\sX)$ is a \emph{mean-field equilibrium} if $\pi \in \Psi(\mu)$ and $\mu \in \Lambda(\pi)$.
\end{definition}

The main goal in average-cost mean-field games is to establish the existence of a mean-field equilibrium. To that end, we impose the assumptions below on the components of the mean-field game model. Note that a function $w: \sX \rightarrow [0,\infty)$ is called a moment function if there exists a non-decreasing sequence of compact sets $K_n \uparrow \sX$ such that
$$
\lim_{n \rightarrow \infty} \inf_{x \notin K_n} w(x) = \infty. 
$$ 

\begin{assumption}\label{as1}
\begin{itemize}
\item [(a)] The cost function $c$ is bounded and continuous.
\item [(b)] The stochastic kernel $p$ is weakly continuous; that is, if $(x_n,a_n,\mu_n) \rightarrow (x,a,\mu)$, then $p(\,\cdot\,|x_n,a_n,\mu_n) \rightarrow p(\,\cdot\,|x,a,\mu)$ weakly.
\item [(c)] $\sA$ is compact.
\item [(d)] There exists a non-degenerate sub-probability measure $\lambda$ on $\sX$ such that
$$p(\,\cdot\,|x,a,\mu) \geq \lambda(\,\cdot\,)$$
for all $x \in \sX$, $a \in \sA$, and $\mu \in \P(\sX)$.
\item [(e)] There exist a constant $\alpha \in (0,1)$  and a continuous moment function $w: \sX \rightarrow [0,\infty)$ such that
\begin{align}
\sup_{(a,\mu) \in \sA \times \P(\sX)} \int_{\sX} w(y) \, p(dy|x,a,\mu) \leq \alpha \, w(x) + \int_{\sX} w(y) \, \lambda(dy). \label{eqqq1}
\end{align}
\end{itemize}
\end{assumption}

\noindent If we define the following sub-stochastic kernel $\hat{p}(\,\cdot\,|x,a,\mu) = p(\,\cdot\,|x,a,\mu) - \lambda(\,\cdot\,)$, then (\ref{eqqq1}) can be written as 
\begin{align}
\sup_{(a,\mu) \in \sA \times \P(\sX)} \int_{\sX} w(y) \, \hat{p}(dy|x,a,\mu) \leq \alpha \, w(x). \label{eqqq2}
\end{align}
Note that condition (e) is so-called `drift inequality' and condition (d) is so-called `minorization' condition, both of which were used in the literature for studying ergodicity of Markov chains (see \cite{HeLa99}, and references therein). These assumptions are quite general for studying average cost stochastic control problems. Indeed, Assumption~\ref{as1}-(d) is true when the transition probability satisfies conditions R1(a) and R1(b) in \cite{HeMoRo91} (see also \cite[Remark 3.3]{HeMoRo91} and references therein for further conditions). For Assumption~\ref{as1}-(e), we refer the reader to the examples in \cite[Section 7.4]{HeLa99} to see under which conditions on the system components Assumption~\ref{as1}-(e) holds.

The main result of this section is the existence of a mean-field equilibrium under Assumption~\ref{as1}.

\begin{theorem}\label{thm:MFE} Under Assumption~1, the mean-field game $\bigl( \sX, \sA, p, c \bigr)$ admits a mean-field equilibrium $(\pi^*,\mu^*)$.
\end{theorem}

It is important to note that mean-field games are not games in the strict sense. They are stochastic control problems subject to a constraint on the distribution of the state at each time step. In other words, we have a single agent and represent the collective behavior of (a large population of) other agents by an exogenous \textit{state-measure} $\mu \in \P(\sX)$. This measure $\mu$ should also be consistent with the state distributions of this single agent when the agent acts optimality. The proof of Theorem~\ref{thm:MFE} is given Section~\ref{main-proof}. To establish the existence of a mean-field equilibrium, we use dynamic programming principle for average-cost criterion, which is stated via average cost optimality equation (ACOE), in addition to fixed point approach that is commonly used in classical game problems.

\section{Finite Player Game}\label{sec2}

The motivation for studying mean-field games comes from the challenges to establish the existence of Nash equilibria for large population stochastic dynamic games with mean-field interactions. More precisely, suppose that we have a discrete-time $N$-agent stochastic game with state space $\sX$ and action space $\sA$. For every $t \in \{0,1,2,\ldots\}$ and every $i \in \{1,2,\ldots,N\}$, let $x^N_i(t) \in \sX$ and $a^N_i(t) \in \sA$ denote the state and the action of Agent~$i$ at time $t$, and let
\begin{align}
e_t^{(N)}(\,\cdot\,) = \frac{1}{N} \sum_{i=1}^N \delta_{x_i^N(t)}(\,\cdot\,) \in \P(\sX) \nonumber
\end{align}
denote the empirical distribution of the states at time $t$ (i.e., \emph{mean-field term}), where $\delta_x\in\P(\sX)$ is the Dirac measure at $x$. The initial states $x^N_i(0)$ are independent and identically distributed according to $\mu_0$, and, for each $t \ge 0$, next states $(x^N_1(t+1),\ldots,x^N_N(t+1))$ are generated according to the probability distribution
\begin{align}
\prod^N_{i=1} p\big(dx^N_i(t+1)\big|x^N_i(t),a^N_i(t),e^{(N)}_t\big). \nonumber 
\end{align}
A \emph{policy} for a generic agent is a stochastic kernel $\pi$ on $\sA$ given $\sX$. The set of all policies for Agent~$i$ is denoted by $\Pi_i$. Let ${\bf \Pi}^{(N)} = \prod_{i=1}^N \Pi_i$. We let ${\boldsymbol \pi}^{(N)} = (\pi^1,\ldots,\pi^N)$, $\pi^i \in \Pi_i$, denote the $N$-tuple of policies for all the agents in the game. Under such $N$-tuple of policies, actions $(a^N_1(t),\ldots,a^N_N(t))$ at each time $t \ge 0$ are generated according to the probability distribution
\begin{align}
\prod^N_{i=1} \pi^i\big(da^N_i(t)\big|x^N_i(t)\big).\nonumber
\end{align}
Note that agents can only use their local states when constructing their control laws. For Agent~$i$, the average cost under the initial distribution $\mu_0$ and $N$-tuple of policies ${\boldsymbol \pi}^{(N)} \in {\bf \Pi}^{(N)}$ is given by
\begin{align}
J_i^{N}({\boldsymbol \pi}^{(N)}) &= \limsup_{T \rightarrow \infty} \frac{1}{T} E^{{\boldsymbol \pi}^{(N)}}\biggl[\sum_{t=0}^{T-1} c(x_{i}^N(t),a_{i}^N(t),e^{(N)}_t)\biggr]. \nonumber
\end{align}
We now define the notion of Nash equilibrium for this game problem as follows.

\begin{definition}
A $N$-tuple of policies ${\boldsymbol \pi}^{(N*)}= (\pi^{1*},\ldots,\pi^{N*})$ constitutes a \emph{Nash equilibrium} if
\begin{align}
J_i^{(N)}({\boldsymbol \pi}^{(N*)}) = \inf_{\pi^i \in \Pi_i} J_i^{(N)}({\boldsymbol \pi}^{(N*)}_{-i},\pi^i) \nonumber
\end{align}
for each $i=1,\ldots,N$, where ${\boldsymbol \pi}^{(N*)}_{-i} = (\pi^{j*})_{j\neq i}$.
\end{definition}

It is known that it is in general prohibitive to establish the existence of Nash equilibrium under decentralized information structure. Moreover, when the number of agents is large, obtaining Nash equilibrium is computationally intractable. Therefore, it is of interest to prove the existence of approximate Nash equilibrium. To that end, we introduce the following solution concept:
\begin{definition}\label{def1}
A $N$-tuple of policies ${\boldsymbol \pi}^{(N*)} \in {\bf \Pi}^{(N)}$ is an \emph{$\varepsilon$-Nash equilibrium} (for a given $\varepsilon > 0$) if
\begin{align*}
J_i^{(N)}({\boldsymbol \pi}^{(N*)}) &\leq \inf_{\pi^i \in \Pi_i} J_i^{(N)}({\boldsymbol \pi}^{(N*)}_{-i},\pi^i) + \varepsilon
\end{align*}
for each $i=1,\ldots,N$.
\end{definition}  
In mean field games, the aim is now to show that the policy $\pi^*$ in the mean-field equilibrium, when adopted by each agent,  is $\varepsilon$-Nash equilibrium for games with sufficiently many agents. To that end, we need to impose additional assumptions on the components of the game model. 

Let $\rho$ denote the following metric on $\P(\sX)$ that metrizes the weak topology:
$$
\rho(\mu,\nu) = \sum_{m=0}^{\infty} 2^{-m} \, \left|\, \int_{\sX} f(x) \, \mu(dx)-\int_{\sX} f(x) \, \nu(dx) \, \right|, 
$$
where $f_m \in C_b(\sX)$ and $\|f_m\| \leq 1$ for all $m$ (see \cite[Theorem 6.6, p. 47]{Par67}). Define the following moduli of continuity:
\begin{align}
\omega_{c}(r) &= \sup_{(x,a) \in \sX\times\sA} \sup_{\substack{\mu,\nu: \\ \rho(\mu,\nu)\leq r}} |c(x,a,\mu) - c(x,a,\nu)|. \nonumber
\end{align}
In addition to Assumption~\ref{as1}, we impose an additional assumption, which is stated below. 

\begin{assumption}\label{as2}
\begin{itemize}
\item [(a)] The transition probability $p(\,\cdot\,|x,a)$ does not depend on $\mu$.
\item [(b)] $\omega_c(r) \rightarrow 0$ as $r\rightarrow0$.
\end{itemize}
\end{assumption}

\noindent For average cost criterion, Assumption~\ref{as2}-(a) is quite common to establish the existence of approximate Nash equilibrium (see \cite{Bis15,Wie19}). 
The following theorem is the main result of this section, which states that the policy ${\boldsymbol \pi}^{(N*)} = (\pi^*,\ldots,\pi^*)$, where $\pi^*$ is repeated $N$ times, which itself is obtained from the mean-field equilibrium, is an $\varepsilon$-Nash equilibrium for sufficiently large $N$.

\begin{theorem}\label{appr-thm}
For any $\varepsilon>0$, there exists $N(\varepsilon)$ such that for $N\geq N(\varepsilon)$, the $N$-tuple of policies ${\boldsymbol \pi}^{(N*)}$ is an $\varepsilon$-Nash equilibrium for the game with $N$ agents.
\end{theorem}

The proof of this theorem is given in Section~\ref{sec4-1}.

\section{Proof of Theorem~\ref{thm:MFE}}\label{main-proof}

Under Assumption~\ref{as1}, for any $(\pi,\mu)$, there exists a unique probability measure $\mu_{\pi,\mu}$ such that
\begin{align}
\mu_{\pi,\mu}(\,\cdot\,) = \int_{\sX \times \sA} p(\,\cdot\,|x,a,\mu) \, \pi(da|x) \, \mu_{\pi,\mu}(dx) \label{invariant}
\end{align}
(see \cite[Theorem 3.3]{Veg03}, \cite[Lemma 3.4]{GoHe95}). Given $\mu \in \P(\sX)$, we define the operator $T_{\mu}: C_b(\sX) \rightarrow C_b(\sX)$ as 
\begin{align}
T_{\mu}\, u(x) &= \min_{a \in \sA} \biggl[ c(x,a,\mu) + \int_{\sX} u(y) \, p(dy|x,a,\mu) - \int_{\sX} u(y) \, \lambda(dy)\biggr], \nonumber \\
&= \min_{a \in \sA} \biggl[ c(x,a,\mu) + \int_{\sX} u(y) \, \hat{p}(dy|x,a,\mu)\biggr]. \nonumber
\end{align}
Under Assumption~\ref{as1}, $T_{\mu}$ is a well-defined operator. One can also prove that $T_{\mu}$ is a contraction operator with modulus $\beta = 1 - \lambda(\sX) \in (0,1)$ \cite[Theorem 3.21]{SaLiYuBook}. Therefore, for each $\mu$, there exists a fixed point $h_{\mu} \in C_b(\sX)$ of $T_{\mu}$ by Banach Fixed Point Theorem. The following theorem characterizes the optimal policies for each state-measure $\mu$ using the operator $T_{\mu}$. 

\begin{theorem}\label{opt-pol}
Given $\mu \in \P(\sX)$, a policy $\pi$ is optimal for $\mu$ when $x(0) \sim \mu_{\pi,\mu}$, if 
\begin{align}
\nu_{\pi,\mu}\biggl( \biggl\{ (x,a): c(x,a,\mu) + \int_{\sX} h_{\mu}(dy) \, \hat{p}(dy|x,a,\mu) = T_{\mu} \, h_{\mu} \biggr\} \biggr) = 1, \nonumber
\end{align}
where $\nu_{\pi,\mu}(dx,da) = \mu_{\pi,\mu}(dx) \, \pi(da|x)$ and $h_{\mu}$ is the fixed point of $T_{\mu}$.
\end{theorem}

\begin{proof}
Note that if $T_{\mu} \, h_{\mu} = h_{\mu}$, then we have
$$
h_{\mu}(x) + \rho_{\mu} = \min_{a \in \sA} \biggl[ c(x,a,\mu) + \int_{\sX} h_{\mu}(dy) \, p(dy|x,a,\mu) \biggr], 
$$
where $\rho_{\mu} = \int_{\sX} h_{\mu}(dx) \, \lambda(dx)$. The last equation is called average cost optimality equation (ACOE) in the literature \cite{HeLa96}. Since 
$$
\lim_{n \rightarrow \infty} \frac{E^{\pi}[h_{\mu}(x(n))]}{n} = 0,
$$
for all $\pi$, as $h$ is bounded, we have \cite[Theorem 5.2.4]{HeLa96}
$$
\rho_{\mu} = \inf_{\pi \in \Pi} J_{\mu}(\pi) =: J_{\mu}^*. 
$$
Suppose that $\pi$ satisfies the hypothesis in the theorem. For each $n \geq 1$ and $\gamma \in \P(\sX)$, we define
$$
J_{\mu,n}(\pi,h_{\mu},\gamma) =  E^{\pi} \biggl[ \sum_{t=0}^{n-1} c(x(t),a(t),\mu) + h_{\mu}(x(n)) \, \bigg|\, x(0) \sim \gamma \biggr].
$$
We claim that, for each $n \geq 1$,  
$$
J_{\mu,n}(\pi,h_{\mu},\mu_{\pi,\mu}) = n \rho_{\mu} + \int_{\sX} h_{\mu}(dx) \, \mu_{\pi,\mu}(dx). \nonumber
$$
Claim clearly holds for $n=0$. Suppose it holds for $n$ and consider $n+1$: 
\begin{align}
J_{\mu,n+1}(\pi,h_{\mu},\mu_{\pi,\mu}) &= \int_{\sX\times\sA} \biggl[ c(x,a,\mu) + \int_{\sX} J_{\mu,n}(\pi,h_{\mu},\delta_y) \, p(dy|x,a,\mu) \biggr] \, \pi(da|x) \, \mu_{\pi,\mu}(dx) \nonumber \\
&\overset{(a)}{=} \int_{\sX\times\sA} c(x,a,\mu) \, \pi(da|x) \, \mu_{\pi,\mu}(dx) + \int_{\sX} J_{\mu,n}(\pi,h_{\mu},\delta_y) \, \mu_{\pi,\mu}(dy) \nonumber \\
&\overset{(b)}{=} \int_{\sX\times\sA} \biggl[ c(x,a,\mu) + \int_{\sX} \bigl[n\rho_{\mu} + h_{\mu}(dy)\bigr] \, p(dy|x,a,\mu) \biggr] \, \pi(da|x) \, \mu_{\pi,\mu}(dx) \nonumber \\
&= n\rho_{\mu} + \int_{\sX\times\sA} \biggl[ c(x,a,\mu) + \int_{\sX} h_{\mu}(dy) \, p(dy|x,a,\mu) \biggr] \, \nu_{\pi,\mu}(dx,da) \nonumber \\
&\overset{(c)}{=} n \rho_{\mu} + \int_{\sX\times\sA} \biggl[T_{\mu} h_{\mu}(x) + \rho_{\mu} \biggr] \, \nu_{\pi,\mu}(dx,da)  \nonumber \\
&\overset{(d)}{=} (n+1)\rho_{\mu} + \int_{\sX\times\sA} h_{\mu}(x) \, \mu_{\pi,\mu}(dx), \nonumber 
\end{align}
where (a) follows from (\ref{invariant}), (b) follows from
\begin{align}
J_{\mu,n}(\pi,h_{\mu},\mu_{\pi,\mu})  &= \int_{\sX} J_{\mu,n}(\pi,h_{\mu},\delta_x) \, \mu_{\pi,\mu}(dx) \nonumber
\end{align}
and induction hypothesis, (c) follows from the hypothesis in the theorem, and (d) follows from $T_{\mu} \, h_{\mu} = h_{\mu}$. This completes the proof of the claim. Note that 
\begin{align}
J_{\mu}(\pi) &= \limsup_{n\rightarrow\infty} \frac{1}{n} J_{\mu,n}(\pi,0,\mu_{\pi,\mu}) \nonumber \\
&= \limsup_{n\rightarrow\infty} \frac{1}{n} \biggl[ J_{\mu,n}(\pi,h_{\mu},\mu_{\pi,\mu}) - E^{\pi}[h_{\mu}(x(n))] \biggr] \nonumber \\
&\overset{(a)}{=} \limsup_{n\rightarrow\infty} \frac{1}{n} J_{\mu,n}(\pi,h_{\mu},\mu_{\pi,\mu}) \nonumber \\ 
&= \limsup_{n\rightarrow\infty} \frac{1}{n} \biggl[ n\rho_{\mu} + \int_{\sX} h_{\mu}(dx) \, \mu_{\pi,\mu}(dx) \biggr] \nonumber \\
&= \rho_{\mu} = J_{\mu}^*, \nonumber
\end{align}
where (a) follows from the fact that $h_{\mu}$ is bounded. Therefore, $\pi$ is optimal. 
\end{proof}

Define the set $\P_c(\sX)$ as 
$$
\P_c(\sX) = \left\{\mu \in \P(\sX): \int_{\sX} w(x) \, \mu(dx) \leq \frac{\int_{\sX} w(x) \, \lambda(dx)}{1-\alpha}\right\}.
$$
Since $w$ is a continuous moment function, $\P_c(\sX)$ is compact \cite[Proposition E.8]{HeLa96}. 
We also define 
\begin{align}
\Xi = \left\{\nu \in \P(\sX\times\sA): \nu(dx,da) = \mu(dx) \, \pi(da|x) \text{ } \text{for some} \text{  } (\pi,\mu) \in \Pi\times\P_c(\sX) \right\}. \nonumber
\end{align}
Note that compactness of $\sA$ and compactness of $\P_c(\sX)$ imply that $\Xi$ is tight
\cite[Definition E.5]{HeLa96}, and so, $\Xi$ is relatively compact \cite[Theorem E.6]{HeLa96}. Since $w$ is continuous, $\Xi$ is closed. Therefore, $\Xi$ is compact. For any $\nu \in \P(\sX \times \sA)$, we let $\nu_1$ denote the marginal distribution on $\sX$.

Note that $(\pi,\mu)$ is a mean-field equilibrium if $\pi \in \Psi(\Lambda(\pi))$. We transform this fixed point equation $\pi \in \Psi(\Lambda(\pi))$  into a fixed point equation of a set-valued mapping from $\Xi$ to $2^{\P(\sX\times\sA)}$. To that end, we define the set-valued mapping $\Gamma: \Xi \rightarrow 2^{\P(\sX\times\sA)}$ as follows:
\begin{align}
\Gamma(\nu) = C(\nu) \cap B(\nu), \nonumber
\end{align}
where
\begin{align}
C(\nu) &= \biggl\{ \nu': \nu'_{1}(\,\cdot\,) = \int_{\sX \times \sA} p(\,\cdot\,|x,a,\nu_{1}) \, \nu(dx,da) \biggr\} \nonumber \\
\intertext{and}
B(\nu) &= \biggl\{ \nu': \nu' \biggl( \biggr\{ (x,a) : c(x,a,\nu_1) + \int_{\sX} h_{\nu_1}(y) \, \hat{p}(dy|x,a,\nu_1) = T_{\nu_1} h_{\nu_1}(x) \biggr\} \biggr) = 1 \biggr\}. \nonumber
\end{align}
Note that the set $C(\nu)$ characterizes the equation $\mu = \Lambda(\pi)$ and the set $B(\nu)$ characterizes the equation $\pi \in \Psi(\mu)$. Under Assumption~\ref{as1}-(e), one can easily prove that the image of $\Xi$ under $\Gamma$ is contained in $2^{\Xi}$; that is, $\Gamma(\bnu) \subset \Xi$. 
Indeed, fix any $\nu \in \Xi$. It is sufficient to prove that $C(\nu) \subset \Xi$. Let $\nu' \in C(\nu)$. We have
\begin{align}
\int_{\sX} w(y) \, \nu'_{1}(dy) &= \int_{\sX \times \sA} \int_{\sX} w(y) \, p(dy|x,a,\nu_{1}) \, \nu(dx,da) \nonumber \\
&\leq \int_{\sX} \alpha w(x) \, \nu_{1}(dx) + \int_{\sX} w(y) \, \lambda(dy) \nonumber \text{ }(\text{by Assumption~1-(e)}) \\
&\leq \frac{\int_{\sX} w(x) \, \lambda(dx)}{1-\alpha} \text{ }(\text{as $\nu_{1} \in \P_c(\sX)$}). \nonumber
\end{align}
Hence, $\nu'_{1} \in \P_c(\sX)$. This completes the proof of the claim. Therefore,
$$
\Gamma: \Xi \rightarrow 2^{\Xi}.
$$
A point $\nu \in \Xi$ is a fixed point of $\Gamma$ if $\nu \in \Gamma(\nu)$. The following proposition makes the connection between mean-field equilibria and the fixed points of $\Gamma$.

\begin{proposition}\label{prop1}
Suppose that $\Gamma$ has a fixed point $\nu$. Disintegrate $\nu$ as $\nu(dx,da) = \nu_1(dx) \, \pi(da|x)$. Then the pair $(\pi,\nu_1)$ is a mean-field equilibrium.
\end{proposition}

\begin{proof}
Let $\nu \in \Gamma(\nu)$. Since $\nu \in C(\nu)$, we have $\Lambda(\pi) = \nu_1 = \mu_{\pi,\nu_1}$. Moreover, since $\nu \in B(\nu)$, $\pi$ is optimal for $\nu_1$ when $x(0) \sim \nu_1$ by Theorem~\ref{opt-pol}; that is, $\pi \in \Psi(\nu_1)$. This completes the proof.
\end{proof}

\noindent By Proposition~\ref{prop1}, it suffices to prove that $\Gamma$ has a fixed point in order to establish the existence of a mean-field equilibrium. To prove this, we use Kakutani's fixed point theorem (\cite[Corollary 17.55]{AlBo06}). Note that $\Xi$ is convex. 
Furthermore, it can be proved that $C(\nu) \cap B(\nu) \neq \emptyset$ for any $\nu \in \Xi$. Indeed, we define
$$
\mu(\,\cdot\,) = \int_{\sX \times \sA} p(\,\cdot\,|x,a,\nu_{1}) \, \nu(dx,da).
$$
Moreover, let $f: \sX \rightarrow \sA$ be the minimizer of the following optimality equation:
\begin{align}
&c(x,f(x),\nu_1) + \int_{\sX} h_{\nu_1}(y) \, \hat{p}(dy|x,f(x),\nu_1) \nonumber \\
&\phantom{xxxxxxxxxxxxxxxxxxxx}= \inf_{a \in \sA} \biggl[ c(x,a,\nu_1) + \int_{\sX} h_{\nu_1}(y) \, \hat{p}(dy|x,a,\nu_1) \biggr]. \nonumber
\end{align}
Existence of such an $f$ follows from the Measurable Selection Theorem \cite[Section D]{HeLa96}. If we define $\nu'(dx,da) = \mu(dx) \, \delta_{f(x)}(da)$, then it is straightforward to prove that $\nu' \in C(\nu) \cap B(\nu)$, and so, $C(\nu) \cap B(\nu) \neq \emptyset$. We can also show that both $C(\nu)$ and $B(\nu)$ are convex, and thus their intersection is also convex. Moreover, $\Xi$ is a convex compact subset of a locally convex topological space $\M(\sX \times \sA)$, where $\M(\sX \times \sA)$ denotes the set of finite signed measures on $\sX \times \sA$. Therefore, to establish the existence of a fixed point of $\Gamma$ via Kakutani's fixed point, we only need to prove the following:

\begin{proposition}\label{prop3}
The graph of $\Gamma$, i.e., the set
	$$
	\Gr(\Gamma) := \left\{ (\nu,\xi) \in \Xi \times \Xi : \xi \in \Gamma(\nu)\right\},
	$$
is closed.
\end{proposition}

\begin{proof}
Here, we adapt the proof of \cite[Proposition 3.9]{SaBaRaSIAM} to the average cost criterion. Let $\bigl\{(\nu^{(n)},\xi^{(n)})\bigr\}_{n\geq1} \subset \Xi \times \Xi$ be such that $\xi^{(n)} \in \Gamma(\nu^{(n)})$ for all $n$ and $(\nu^{(n)},\xi^{(n)}) \rightarrow (\nu,\xi)$ as $n\rightarrow\infty$ for some $(\nu,\xi) \in \Xi \times \Xi$. We prove that $\xi \in \Gamma(\nu)$, which completes the proof.

We first prove that $\xi \in C(\nu)$. Note that, for all $n$, we have
\begin{align}
\xi^{(n)}_{1}(\,\cdot\,) = \int_{\sX \times \sA} p(\,\cdot\,|x,a,\nu^{(n)}_{1}) \, \nu^{(n)}(dx,da). \label{eqc}
\end{align}
Let $g \in C_b(\sX)$. Then, by \cite[Theorem 3.5]{Lan81}, we have
\begin{align}
\lim_{n\rightarrow\infty} \int_{\sX \times \sA} \int_{\sX}& g(y) \,  p(dy|x,a,\nu^{(n)}_{1}) \, \nu^{(n)}(dx,da) 
=\int_{\sX \times \sA} \int_{\sX} g(y) \, p(dy|x,a,\nu_{1})\, \nu(dx,da) \nonumber
\end{align}
since $\nu^{(n)} \rightarrow \nu$ weakly and $\int_{\sX} g(y) \, p(dy|x,a,\nu^{(n)}_{1})$ converges to $\int_{\sX} g(y) \, p(dy|x,a,\nu_{1})$ continuously\footnote{Suppose $g$, $g_n$ ($n\geq1$) are measurable functions on metric space $\sE$. The sequence $g_n$ is said to converge to $g$ continuously if $\lim_{n\rightarrow\infty}g_n(e_n)=g(e)$ for any $e_n\rightarrow e$ where $e \in \sE$.} (see \cite[p. 388]{Ser82}). This implies that the sequence of measures on the right-hand side of (\ref{eqc}) converges weakly to $\int_{\sX \times \sA} p(\,\cdot\,|x,a,\nu_{1}) \, \nu(dx,da)$. Therefore, we have
\begin{align}
\xi_{1}(\,\cdot\,) = \int_{\sX \times \sA} p(\,\cdot\,|x,a,\nu_{1}) \, \nu(dx,da), \nonumber
\end{align}
which means that $\xi \in C(\nu)$.

We now prove that $\xi \in B(\nu)$, which implies that $\xi \in \Gamma(\nu)$. To prove this, we use the continuity properties of the average cost optimality equation. For each $n$, let us define
\begin{align}
F^{(n)}(x,a) &= c(x,a,\nu_1^{(n)}) + \int_{\sX} h_{\nu_1^{(n)}}(y) \, \hat{p}(dy|x,a,\nu_1^{(n)}) \nonumber \\
\intertext{and}
F(x,a) &= c(x,a,\nu_1) + \int_{\sX} h_{\nu_1}(y) \, \hat{p}(dy|x,a,\nu_1). \nonumber
\end{align}
By definition,
\begin{align}
h_{\nu_1^{(n)}}(x) = \inf_{a \in \sA} F^{(n)}(x,a) \text{ } \text{ and } \text{ } h_{\nu_1}(x) = \inf_{a \in \sA} F(x,a). \nonumber
\end{align}
By assumption, we have
\begin{align}
1 = \xi^{(n)}\biggl( \biggl\{ (x,a): F^{(n)}(x,a) = h_{\nu_1^{(n)}}(x) \biggr\} \biggr), \text{ } \text{for all $n$}. \nonumber
\end{align}
Let $A^{(n)} = \bigl\{ (x,a): F^{(n)}(x,a) = h_{\nu_1^{(n)}}(x)\bigr\}$. Since both $F^{(n)}$ and $h_{\nu_1^{(n)}}$ are continuous, $A^{(n)}$ is closed. Define $A = \bigl\{ (x,a): F(x,a) = h_{\nu_1}(x) \bigr\}$ which is also closed as both $F$ and $h_{\nu_1}$ are continuous.

We first prove $\xi \in B(\nu)$ by supposing that $F^{(n)}$ converges to $F$ continuously and $h_{\nu_1^{(n)}}$ converges to $h_{\nu_1}$ continuously, as $n\rightarrow\infty$. Then, we prove these statements.

For each $M\geq1$, define closed set $B^M = \bigl\{ (x,a): F(x,a) \geq h_{\nu_1}(x) + \epsilon(M) \bigr\}$, where $\epsilon(M) \rightarrow 0$ as $M\rightarrow \infty$. Since both $F$ and $h_{\nu_1}$ are continuous, we can choose $\{\epsilon(M)\}_{M\geq1}$ so that $\xi(\partial B^M) = 0$ for each $M$. As $A^c = \bigcup_{M=1}^{\infty} B^M$ and $B^M \subset B^{M+1}$, we have by the monotone convergence theorem
\begin{align}
\xi^{(n)}\big(A^c \cap A^{(n)}\big) = \liminf_{M\to\infty} \xi^{(n)}\big(B^M \cap A^{(n)}). \nonumber
\end{align}
Hence, we have
\begin{align}
1 &= \limsup_{n\rightarrow\infty} \liminf_{M\rightarrow\infty} \biggl\{ \xi^{(n)}\big(A \cap A^{(n)}\big) + \xi^{(n)}\big(B^M \cap A^{(n)}\big)\biggr\} \nonumber\\
&\leq \liminf_{M\rightarrow\infty} \limsup_{n\rightarrow\infty}  \biggl\{\xi^{(n)}\big(A \cap A^{(n)}\big) + \xi^{(n)}\big(B^M \cap A^{(n)}\big)\biggr\}. \nonumber
\end{align}
\noindent Fix any $M$. Note that $\xi^{(n)}$ converges weakly to $\xi$ as $n\rightarrow\infty$ when both measures are restricted to $B^M$, as $B^M$ is closed and $\xi(\partial B^M)=0$ (see \cite[Theorem 8.2.3]{Bog07}). Furthermore, $1_{A^{(n)} \cap B^M}$ converges continuously to $0$. Indeed, if $(x^{(n)},a^{(n)}) \rightarrow (x,a)$ in $B^M$, then
\begin{align}
\lim_{n\rightarrow\infty} F^{(n)}(x^{(n)},a^{(n)}) &= F(x,a) \nonumber \\
&\geq h_{\nu_1}(x) + \epsilon(M) \nonumber \\
&= \lim_{n\rightarrow\infty} h_{\nu_1^{(n)}}(x^{(n)}) + \epsilon(M). \nonumber
\end{align}
Hence, for large enough $n$, we have $F^{(n)}(x^{(n)},a^{(n)}) > h_{\nu_1^{(n)}}(x^{(n)})$, which implies that $(x^{(n)},a^{(n)}) \not\in A^{(n)}$, and so, $1_{A^{(n)} \cap B^M}(x^{(n)},a^{(n)}) = 0$. Therefore, by \cite[Theorem 3.5]{Lan81}, for each $M$, we have
\begin{align*}
\limsup_{n\rightarrow\infty} \xi^{(n)}\big(B^M \cap A^{(n)}\big) = 0.
\end{align*}
Therefore, we obtain
\begin{align*}
1 &\leq  \limsup_{n\rightarrow\infty} \xi^{(n)}\big(A \cap A^{(n)}\big)\\
&\le \limsup_{n\rightarrow\infty} \xi^{(n)}(A) \nonumber \\
&\leq \xi(A), \nonumber
\end{align*}
where the last inequality follows from the Portmanteau Theorem (see \cite[Theorem 2.1]{Bil99}). Hence, $\xi(A)=1$. This means that $\xi \in B(\nu)$. Therefore, $\xi \in \Gamma(\nu)$ which completes the proof under the condition that $F^{(n)}$ converges to $F$ continuously and $h_{\nu_1^{(n)}}$ converges to $h_{\nu_1}$ continuously, as $n\rightarrow\infty$, which we prove next.

For continuous functions, continuous convergence is the same as uniform convergence on compact sets (see \cite[Lemma 2.1]{Lan81}). Therefore, we show that $F^{(n)}$ uniformly converges to $F$ over compact sets and $h_{\nu_1^{(n)}}$ uniformly converges to $h_{\nu_1}$ over compact sets, as these functions are continuous.
Furthermore, if $h_{\nu_1^{(n)}}$ converges to $h_{\nu_1}$ continuously, then $F^{(n)}$ also converges to $F$ continuously. Indeed, let $(x^{(n)},a^{(n)}) \rightarrow (x,a)$. By \cite[Theorem 3.5]{Lan81} we have
\begin{align}
\lim_{n\rightarrow\infty} F^{(n)}(x^{(n)},a^{(n)}) &= \lim_{n\rightarrow\infty} \biggl[ c(x,a,\nu_1^{(n)}) + \int_{\sX} h_{\nu_1^{(n)}}(y) \, \hat{p}(dy|x,a,\nu_1^{(n)}) \biggr]\nonumber \\
&= \biggl[ c(x,a,\nu_1) + \int_{\sX} h_{\nu_1}(y) \, \hat{p}(dy|x,a,\nu_1) \biggr]\nonumber \\
&= F(x,a). \nonumber
\end{align}
Therefore, it is enough to prove that $h_{\nu_1^{(n)}}$ uniformly converges to $h_{\nu_1}$ over compact sets. The following proposition establishes this result. 

\begin{proposition}\label{prop4}
For any compact $K\subset \sX$, we have
\begin{align}
\lim_{n\rightarrow\infty} \sup_{x \in K} \bigl| h_{\nu_1^{(n)}}(x) - h_{\nu_1}(x)| = 0. \nonumber
\end{align}
\end{proposition}

\begin{proof}
To ease the notation, let $T^{(n)} = T_{\nu_1^{(n)}}$, $h^{(n)} = h_{\nu_1^{(n)}}$, $T = T_{\nu_1}$, and $h = h_{\nu_1}$. Let $u^{(n)}_0 = u_0 = 0$ and
\begin{align}
u^{(n)}_{k+1} = T^{(n)} u^{(n)}_k \text{ } \text{and} \text{ } u_{k+1} = T u_k. \nonumber
\end{align}
Since $T^{(n)}$ and $T$ are contractive operators on $C_b(\sX)$ with modulus $\beta$, it can be proved that
\begin{align}
\| u^{(n)}_k - h^{(n)} \| , \, \| u_k - h \|  \leq L_0 \, \beta^k, \label{eq7}
\end{align}
for some constant $L_0$.

For any compact $K \subset \sX$, one can prove that 
\begin{align}
\lim_{n\rightarrow\infty} \sup_{x \in K} \bigl| u^{(n)}_{k}(x) - u_{k}(x) \bigr| = 0, \label{compactconv}
\end{align}
for all $k\geq0$. Indeed, fix any compact $K \subset \sX$. Since $u^{(n)}_0 = u_0 = 0$, the claim trivially holds for $k=0$. Suppose that the claim holds for $k$. Then, consider $k+1$:
\begin{align}
&\sup_{x \in K} \bigl| u^{(n)}_{k+1}(x) - u_{k+1}(x) \bigr| \nonumber \\
&\phantom{xxxxxx}=\sup_{x \in K} \biggl| \min_{a \in \sA} \biggl[ c(x,a,\nu^{(n)}_{1}) + \int_{\sX} u^{(n)}_{k}(y) \, \hat{p}(dy|x,a,\nu^{(n)}_{1}) \biggr] \nonumber \\
&\phantom{xxxxxxxxxxxxxxxx}- \min_{a \in \sA}\biggl[ c(x,a,\nu_{1}) + \beta \int_{\sX} u_{k}(y) \, \hat{p}(dy|x,a,\nu_{1}) \biggr] \biggr| \nonumber \\
&\phantom{xxxxxx}\leq \sup_{(x,a) \in K \times \sA} \bigl| c(x,a,\nu_{1}^{(n)}) - c(x,a,\nu_{1}) \bigr| \nonumber \\
&\phantom{xxxxxxxxxxxxxxxx}+ \sup_{(x,a) \in K \times \sA} \biggl| \int_{\sX} u^{(n)}_{k}(y) \, \hat{p}(dy|x,a,\nu^{(n)}_{1}) \nonumber \\
&\phantom{xxxxxxxxxxxxxxxxxxxxxxxxxxxxxxxxx}- \int_{\sX} u_{k}(y) \, \hat{p}(dy|x,a,\nu_{1}) \biggr|. \nonumber
\end{align}
Note that $c(\cdot\,,\,\cdot\,,\nu^{(n)})$ converges to $c(\cdot\,,\,\cdot\,,\nu)$ continuously as $n\rightarrow\infty$. Furthermore,  since $u^{(n)}_{k}$ converges to $u_{k}$ continuously, by \cite[Theorem 3.5]{Lan81} $\int_{\sX} u^{(n)}_{k}(y) \hat{p}(dy|\cdot\,,\,\cdot\,,\nu^{(n)}_{1})$ converges to $\int_{\sX} u_{k}(y) \hat{p}(dy|\cdot\,,\,\cdot\,,\nu_{1})$ continuously as $n\rightarrow\infty$. Since continuous convergence is equivalent to uniform convergence over compact sets for continuous functions, the last expression goes to zero as $n\rightarrow\infty$. This proves (\ref{compactconv}).

Using (\ref{compactconv}), we now complete the proof of Proposition~\ref{prop4}. Fix any compact $K \subset \sX$. For all $k\geq0$, we have
\begin{align}
\sup_{x \in K} \bigl| h^{(n)}(x) - h(x)| &\leq \| h^{(n)} - u^{(n)}_{k}\| + \sup_{x \in K} \bigl| u^{(n)}_{k}(x) - u_{k}(x) \bigr| + \|u_{k} - h\| \nonumber \\
&\leq 2 L_0  \beta^k + \sup_{x \in K} \bigl| u^{(n)}_{k}(x) - u_{k}(x) \bigr| \text{ } \text{ (by (\ref{eq7}))}. \nonumber
\end{align}
This last expression can be made arbitrarily small by first choosing large enough $k$ and then large enough $n$. 
\end{proof}
This completes the proof of Proposition~\ref{prop3} as $F^{(n)}$ converges to $F$ continuously and $h_{\nu_1^{(n)}}$ converges to $h_{\nu_1}$ continuously by Proposition~\ref{prop4}. 
\end{proof}

Recall that $\Xi$ is a compact convex subset of the locally convex topological space $\M(\sX \times \sA)$. Furthermore, $\Gamma$ has closed graph by Proposition~\ref{prop3}, and it takes nonempty convex values. Therefore, by Kakutani's fixed point theorem (see \cite[Corollary 17.55]{AlBo06}), $\Gamma$ has a fixed point. Then, Theorem~\ref{thm:MFE} follows from Proposition~\ref{prop1}.

\section{Proof of Theorem~\ref{appr-thm}}\label{sec4-1}

To prove Theorem~\ref{appr-thm}, we mimic the idea used in \cite{Wie19} of formulating $N$-agent game as a stochastic control problem over an extended state space and action space. Namely, we consider $N$-agent game as a stochastic control problem with state space $\sX^N = \prod_{i=1}^N \sX$, action space $\sA^N = \prod_{i=1}^N \sA$, transition probability
\begin{align}
P^N(dy_1,\ldots,dy_N|x_1,\ldots,x_N,a_1,\ldots,a_N) = \prod_{i=1}^N p(dy_i|x_i,a_i), \nonumber
\end{align}
and one-stage cost function
\begin{align}
C^N(x_1,\ldots,x_N,a_1,\ldots,a_N) = c\left(x_1,a_1,\frac{1}{N}\sum_{i=1}^N \delta_{x_i}\right). \nonumber 
\end{align}
Here, the initial state $(x_1(0),\ldots,x_N(0))$ is generated by the product measure $\mu_0^N = \prod_{i=1}^N \mu_0$. For any $N$-tuple of policies ${\boldsymbol \pi}^{(N)}$, the average-cost of this stochastic control problem is the same as the average-cost of Agent~1 in the game; that is,
$$
J_1^N({\boldsymbol \pi}^{(N)}) = \limsup_{T\rightarrow\infty} \frac{1}{T} E\biggl[ \sum_{t=0}^{T-1} C^N\bigl(x_1^N(t),\ldots,x_N^N(t),a_1^N(t),\ldots,a_N^N(t)\bigr) \biggr]. \nonumber
$$
Suppose that Assumptions~\ref{as1} and \ref{as2} hold.
Then, for all $N$, the following are satisfied:
\begin{itemize}
\item [(i)] The one-stage cost function $C^N$ is bounded and continuous.
\item [(ii)] The stochastic kernel $P^N$ is weakly continuous and satisfies 
\end{itemize}
\begin{align}
&P^N(\,\cdot\,|x_1,\ldots,x_N,a_1,\ldots,a_N) \geq \lambda^N(\,\cdot\,) = \prod_{i=1}^N \lambda(\,\cdot\,), \text{ } \forall \text{ } (x_1,\ldots,x_N) \in \sX^N,(a_1,\ldots,a_N) \in \sA^N \label{Nmin} \\
&\sup_{(a_1,\ldots,a_N)} \int_{\sX^N} w^N(y_1,\ldots,y_N) \, \hat{p}^N(dy_1,\ldots,y_N|x_1,\ldots,x_N,a_1,\ldots,a_N) \leq \alpha^N \, w^N(x_1,\ldots,x_N), \text{ } \forall \text{ } (x_1,\ldots,x_N) \in \sX^N\label{Ndrift}
\end{align}
where $w^N(x_1,\ldots,x_N) = \prod_{i=1}^N w(y_i)$ and 
$\hat{p}^N(\,\cdot\,|x_1,\ldots,x_N,a_1,\ldots,a_N) = p(\,\cdot\,|x_1,\ldots,x_N,a_1,\ldots,a_N) - \lambda^N(\,\cdot\,)$. Under (\ref{Nmin}) and (\ref{Ndrift}), for any ${\boldsymbol \pi}^{(N)}$, the following stochastic kernel on $\sX^N$ given $\sX^N$ 
$$
P^{{\boldsymbol \pi}^{(N)}}(\,\cdot\,|x_1,\ldots,x_N) = \int_{\sA^N} P^N(\,\cdot\,|x_1,\ldots,x_N,a_1,\ldots,a_N) \, \prod_{i=1}^N \pi(da_i|x_i) 
$$
has an unique invariant distribution $\mu_{{\boldsymbol \pi}^{(N)}} \in \P(\sX^N)$ and we have 
$$
J_1^N({\boldsymbol \pi}^{(N)}) = \int_{\sX^N \times \sA^N} c\left(x_1,a_1,\frac{1}{N} \sum_{i=1}^N \delta_{x_i}\right) \, \prod_{i=1}^N \pi_i(da_i|x_i) \, \mu_{{\boldsymbol \pi}^{(N)}}(dx_1,\ldots,dx_N). \nonumber 
$$
(see \cite[Theorem 3.3]{Veg03}, \cite[Lemma 3.4]{GoHe95}). Moreover, 
$$\mu_{{\boldsymbol \pi}^{(N)}}(dx_1,\ldots,dx_N) = \prod_{i=1}^N \mu_{\pi_i}(dx_i),$$
where, for each $i=1,\ldots,N$, $\mu_{\pi_i}$ is the unique invariant probability measure of the stochastic kernel
$$
\int_{\sA} p(\,\cdot\,|x,a) \, \pi_i(da|x). 
$$ 
The following result is the key to prove Theorem~\ref{appr-thm}.

\begin{proposition}\label{theorem3}
For any $\pi \in \Pi_1$, we have
\begin{align}
\bigl| J_{\mu^*}(\pi) - J_1^N\left(\pi,{\boldsymbol \pi}^{(N*)}_{-1}\right) \bigr| \leq \epsilon(N), \nonumber
\end{align}
where 
$$
\epsilon(N) = \int_{\sX^{N-1} \times \sA^{N-1}} \omega_c\left(\rho\left(\mu^*,\frac{1}{N-1}\sum_{i=1}^N \delta_{x_i}\right)\right) \, \mu^*(dx_2) \ldots \mu^*(dx_N) + \omega_c\left(2 \cdot \left(1-\frac{N-1}{N}+\frac{1}{N}\right)\right). \nonumber
$$
Note that, by law of large numbers and the fact that $\omega_c(r) \rightarrow 0$ as $r\rightarrow0$, $\epsilon(N)\rightarrow0$ as $N\rightarrow\infty$.
\end{proposition}

\begin{proof}
Fix any policy $\pi$ for Agent~1. Since $\mu_{\pi}$ is the unique invariant distribution of the stochastic kernel $\int_{\sA} p(\,\cdot\,|x,a) \, \pi(da|x)$, we can write 
$$J_{\mu^*}(\pi) = \int_{\sX \times \sA} c(x_1,a_1,\mu^*) \, \pi(da_1|x_1) \, \mu_{\pi}(dx_1)$$ 
(see \cite[Theorem 3.3]{Veg03}, \cite[Lemma 3.4]{GoHe95}). Then, we have 
\begin{align}
&\bigl| J_{\mu^*}(\pi) - J_1^N\left(\pi,{\boldsymbol \pi}^{(N*)}_{-1}\right) \bigr| \nonumber \\
&= \biggl| \, \int_{\sX \times \sA} c(x_1,a_1,\mu^*) \, \pi(da_1|x_1) \, \mu_{\pi}(dx_1) \nonumber \\
&- \int_{\sX^N\times\sA^N} c\left(x_1,a_1,\frac{1}{N}\sum_{i=1}^N \delta_{x_i}\right) \, \pi(da_1|x_1) \, \prod_{i=2}^N \pi^*(da_i|x_i) \, \mu_{\pi,{\boldsymbol \pi}^{(N*)}_{-1}}(dx_1,\ldots,dx_N) \, \biggr| \nonumber \\
&\overset{(a)}{=} \biggl| \, \int_{\sX \times \sA} c(x_1,a_1,\mu^*) \, \pi(da_1|x_1) \, \mu_{\pi}(dx_1) \nonumber \\
&- \int_{\sX\times\sA} \biggl[ \int_{\sX^{N-1}} c\left(x_1,a_1,\frac{\delta_{x_1}}{N}+\frac{N-1}{N}\frac{1}{N-1}\sum_{i=2}^N \delta_{x_i}\right) \, \mu^*(dx_2)\ldots\mu^*(dx_N) \biggr] \pi(da_1|x_1) \, \mu_{\pi}(dx_1) \, \biggr| \nonumber\\
&\leq \biggl| \, \int_{\sX \times \sA} c(x_1,a_1,\mu^*) \, \pi(da_1|x_1) \, \mu_{\pi}(dx_1) \nonumber \\
&- \int_{\sX\times\sA} \biggl[ \int_{\sX^{N-1}} c\left(x_1,a_1,\frac{1}{N-1}\sum_{i=2}^N \delta_{x_i}\right) \, \mu^*(dx_2)\ldots\mu^*(dx_N) \biggr] \pi(da_1|x_1) \, \mu_{\pi}(dx_1) \, \biggr| \nonumber\\
&+\biggl| \, \int_{\sX\times\sA} \biggl[ \int_{\sX^{N-1}} c\left(x_1,a_1,\frac{1}{N-1}\sum_{i=2}^N \delta_{x_i}\right) \, \mu^*(dx_2)\ldots\mu^*(dx_N) \biggr] \pi(da_1|x_1) \, \mu_{\pi}(dx_1) \nonumber \\
&- \int_{\sX\times\sA} \biggl[ \int_{\sX^{N-1}} c\left(x_1,a_1,\frac{\delta_{x_1}}{N}+\frac{N-1}{N}\frac{1}{N-1}\sum_{i=2}^N \delta_{x_i}\right) \, \mu^*(dx_2)\ldots\mu^*(dx_N) \biggr] \pi(da_1|x_1) \, \mu_{\pi}(dx_1) \, \biggr| \nonumber \\
&\overset{(b)}{\leq} \int_{\sX^{N-1} } \omega_c\left(\rho(\mu^*,\frac{1}{N-1}\sum_{i=2}^N \delta_{x_i})\right) \, \mu^*(dx_2) \ldots \mu^*(dx_N) + \omega_c\left(2 \cdot \left(1-\frac{N-1}{N}+\frac{1}{N}\right)\right), \nonumber
\end{align}
where (a) follows from the fact that $\mu^*$ is the unique invariant probability measure of the stochastic kernel $\int_{\sA} p(\,\cdot\,|x,a) \, \pi^*(da|x)$ and $\mu_{\pi}$ is the unique invariant probability measure of the stochastic kernel $\int_{\sA} p(\,\cdot\,|x,a) \, \pi(da|x)$, and (b) follows from 
$$
\rho\left(\frac{1}{N-1}\sum_{i=2}^N \delta_{x_i},\frac{\delta_{x_1}}{N}+\frac{N-1}{N}\frac{1}{N-1}\sum_{i=2}^N \delta_{x_i}\right) = 2 \cdot \left(1-\frac{N-1}{N}+\frac{1}{N}\right).
$$
This completes the proof.
\end{proof}

Now, we are ready to prove Theorem~\ref{appr-thm}. We prove that, for sufficiently large $N$, we have
\begin{align}
J_i^{N}({\boldsymbol \pi}^{(N)}) &\leq \inf_{\pi \in \Pi_i} J_i^{N}({\boldsymbol \pi}^{(N)}_{-i},\pi) + \varepsilon \label{eq13}
\end{align}
for each $i=1,\ldots,N$. Since the transition probabilities and the one-stage cost functions are the same for all agents in the game, it is sufficient to prove (\ref{eq13}) for Agent~$1$ only. Given $\epsilon > 0$, for each $N\geq1$, let $\tpi^{(N)} \in \Pi_1$ be such that
\begin{align}
J_1^{N} (\tpi^{(N)},\pi^*,\ldots,\pi^*) < \inf_{\pi \in \Pi_1} J_1^{N} (\pi,\pi^*,\ldots,\pi^*) + \frac{\varepsilon}{3}. \nonumber
\end{align}
Let $N(\varepsilon)$ be such that $\epsilon(N) \leq \frac{\varepsilon}{3}$ for all $N \geq N(\varepsilon)$. Then, by Proposition~\ref{theorem3}, we have
\begin{align}
\inf_{\pi \in \Pi_1} J_1^{N} (\pi,\pi^*,\ldots,\pi^*) + \varepsilon &> J_1^{N} (\tpi^{(N)},\pi^*,\ldots,\pi^*) + \frac{2\varepsilon}{3} \nonumber \\
&\geq J_{\mu^*}(\tpi^{(N)}) + \frac{\varepsilon}{3}\nonumber \\
&\geq \inf_{\pi \in \Pi} J_{\mu^*}(\pi) + \frac{\varepsilon}{3} \nonumber \\
&= J_{\mu^*}(\pi^*) + \frac{\varepsilon}{3}\nonumber \\
&\geq J_1^{N} (\pi^*,\pi^*,\ldots,\pi^*). \nonumber
\end{align}
This completes the proof of Theorem~\ref{appr-thm}. 

\section{Conclusion}\label{conc}

This paper has considered discrete-time mean-field games subject to average cost criteria, for Polish state and action spaces. Under drift and minorization conditions on the state transition kernel, the existence of a mean-field equilibrium has been established for infinite-population game models using the average cost optimality equation. We have then applied the policy in mean-field equilibrium to the finite population game and have proved that it constitutes an approximate Nash equilibrium for games with a sufficiently large number of agents.

One interesting future direction of research to pursue is to study partially observed version of the same problem. By converting partially-observed system into a fully-observed one over a belief-state space, it might be possible to establish similar results for the partially-observed case.

\section*{Acknowledgment}
This research was supported by The Scientific and Technological Research Council of Turkey (T\"{U}B\.{I}TAK) B\.{I}DEB 2232 Research Grant. The author is grateful to Professor Tamer Ba{\c{s}}ar and Professor Maxim Raginsky for their constructive comments.

\bibliographystyle{apa}

\end{document}